\documentclass[a4paper]{article}
\usepackage{enumitem}
\usepackage{calc}

\usepackage{lmodern} %Type1-font for non-english texts and characters
\usepackage{graphicx} %%For loading graphic files

\usepackage{amsmath}
\usepackage{amssymb,tikz}
\usepackage{pict2e}
\usepackage{amsthm}
\usepackage{amscd}
\usepackage{mathrsfs}
\usepackage{todonotes}
\usetikzlibrary{calc} % This is to add two points in tikzpictures

\usepackage{yfonts}

\usepackage{marvosym}
\usepackage[percent]{overpic}

\newtheorem{theorem}{Theorem}

\newtheorem{remark}[theorem]{Remark}
\newtheorem{conjecture}[theorem]{Conjecture}
\theoremstyle{definition}

\def\XXint#1#2#3{{\setbox0=\hbox{$#1{#2#3}{\int}$}
\vcenter{\hbox{$#2#3$}}\kern-.5\wd0}}

\usepackage{tikz}
\usetikzlibrary{arrows}
\usetikzlibrary{decorations.pathmorphing}
\usetikzlibrary{decorations.markings}
\usetikzlibrary{patterns}
\usetikzlibrary{automata}
\usetikzlibrary{positioning}
\usepackage{tikz-cd}
\tikzset{->-/.style={decoration={
				markings,
				mark=at position #1 with {\arrow{latex}}},postaction={decorate}}}
	
	\tikzset{-<-/.style={decoration={
				markings,
				mark=at position #1 with {\arrowreversed{latex}}},postaction={decorate}}}

\usetikzlibrary{shapes.misc}\tikzset{cross/.style={cross out, draw, 
         minimum size=2*(#1-\pgflinewidth), 
         inner sep=0pt, outer sep=0pt}}

\usepackage{pgfplots}
\def\bigO{{\cal O}}

\usepackage[colorlinks=true]{hyperref}
\hypersetup{urlcolor=blue, citecolor=red, linkcolor=blue}

\begin{document}
%\title{The constant term for large gap asymptotics at the hard edge for Muttalib--Borodin ensembles}
%\title{Large gap asymptotics at the hard edge for Muttalib--Borodin ensembles: Higher order terms}
\title{Upper bounds for the maximum deviation \\ of the Pearcey process}
\author{Christophe Charlier}

\maketitle

\begin{abstract}
The Pearcey process is a universal point process in random matrix theory and depends on a parameter $\rho \in \mathbb{R}$. Let $N(x)$ be the random variable that counts the number of points in this process that fall in the interval $[-x,x]$. In this note, we establish the following global rigidity upper bound:
\begin{align*}
\lim_{s \to \infty}\mathbb P\left(\sup_{x> s}\left|\frac{N(x)-\big( \frac{3\sqrt{3}}{4\pi}x^{\frac{4}{3}}-\frac{\sqrt{3}\rho}{2\pi}x^{\frac{2}{3}} \big)}{\log x}\right| \leq  \frac{4\sqrt{2}}{3\pi} + \epsilon \right) = 1,
\end{align*}
where $\epsilon > 0$ is arbitrary. We also obtain a similar upper bound for the maximum deviation of the points, and a central limit theorem for the individual fluctuations. The proof is short and combines a recent result of Dai, Xu and Zhang with another result of Charlier and Claeys.
\end{abstract}

\noindent
{\small{\sc AMS Subject Classification (2010)}: 41A60, 60B20, 33B15, 33E20, 35Q15.}

\noindent
{\small{\sc Keywords}: Pearcey process, Rigidity, Random matrix theory.}

%\tableofcontents

\subsection*{Introduction and statement of results}
\label{Section:intro}
The Pearcey process describes the local eigenvalue statistics of large random matrices near the points of the spectrum where the limiting mean eigenvalue density vanishes as a cubic root. This has been established rigorously for random matrix ensembles with an external source \cite{BreHik1, BreHik2, TradWidPearcey, BleherKuijlaarsIII}, a two-matrix model \cite{GeuZhang}, large complex correlated Wishart matrices \cite{HHN2}, and for general complex Hermitian Wigner-type matrices \cite{ErdorKrugerSchroder}. Beyond matrix ensembles, the Pearcey process also appears in Brownian motion models \cite{AdlerOrantinMoerbeke, AdlerMoerbeke, BleherKuijlaarsIII} and in random plane partition models \cite{OkounkovReshetikhin}.

\medskip The Pearcey process is a determinantal point process on $\mathbb{R}$ whose kernel is given by
\begin{align*}
K^{\mathrm{Pe}}_{\rho}(x,y) = \frac{1}{(2\pi)^{2}} \int_{0}^{\infty} \left( \int_{-\infty}^{\infty}e^{-\frac{1}{4}t^{4}-\frac{\rho}{2}t^{2}+it(x+z)}dt \right) \left( \int_{\Sigma}e^{\frac{1}{4}t^{4}+\frac{\rho}{2}t^{2}+it(y+z)}dt \right)dz, 
\end{align*}
where $\rho \in \mathbb{R}$ is a parameter of the model and $\Sigma$ consists of four rays: $\Sigma = (e^{\frac{\pi i}{4}}\infty,0)\cup(0,e^{\frac{3\pi i}{4}}\infty) \cup (e^{-\frac{3\pi i}{4}}\infty,0)\cup(0,e^{-\frac{\pi i}{4}}\infty)$. If $\rho \to + \infty$, the Pearcey process favors the point configurations with fewer points near $0$, and in this case the large gap asymptotics are closely related to the Airy gap probabilities \cite[Theorem 5.1]{BertolaCafasso}. Large gap asymptotics and first exponential moment asymptotics for any fixed $\rho$ have only recently been obtained by Dai, Xu and Zhang \cite{DXZ2020,DXZ2020 thinning}.

\medskip In this note, we obtain a central limit theorem (CLT) and establish two global rigidity upper bounds for the Pearcey process. For $x \geq 0$, let $N(x)$ denote the random variable that counts the number of points in the Pearcey process that fall in the interval $[-x,x]$, and let $x_{k} \geq 0$ denote the smallest number such that $N(x_{k})=k$. All the points $\{x_{k}\}_{k\geq 1}$ almost surely exist, and by definition they satisfy $0 < x_{1} < x_{2} < x_{3} < \ldots$. %The points of the Pearcey process are almost surely all distinct, which implies that  %By definition, 

\newpage

\begin{theorem}\label{thm:rigidity Pearcey}%{\bf (Rigidity for the Pearcey process)}
As $k \to + \infty$, we have
\begin{align}\label{convergence a la gustavsson}
\pi\frac{\frac{3\sqrt{3}}{4\pi}x_k^{\frac{4}{3}}-\frac{\sqrt{3}\rho}{2\pi}x_k^{\frac{2}{3}}-k}{\sqrt{\log k}} \quad  \overset{d}{\longrightarrow} \quad  \mathcal{N}(0,1)
\end{align}
where $\overset{d}{\longrightarrow}$ means convergence in distribution and $\mathcal{N}(0,1)$ is a zero-mean normal random variable with variance $1$. Furthermore, for any $\epsilon > 0$, we have
\begin{align}
& \lim_{s \to \infty}\mathbb P\left(\sup_{x> s}\left|\frac{N(x)-\big( \frac{3\sqrt{3}}{4\pi}x^{\frac{4}{3}}-\frac{\sqrt{3}\rho}{2\pi}x^{\frac{2}{3}} \big)}{\log x}\right| \leq  \frac{4\sqrt{2}}{3\pi} + \epsilon \right) = 1, \label{upper bound rigidity 1} \\
& \lim_{k_0\to\infty}\mathbb P\left( \sup_{k \geq k_{0}} \frac{|\frac{3\sqrt{3}}{4\pi}x_k^{\frac{4}{3}}-\frac{\sqrt{3}\rho}{2\pi}x_k^{\frac{2}{3}}-k|}{\log k}\leq \frac{\sqrt{2}}{\pi} + \epsilon\right)=1. \label{upper bound rigidity 2}
\end{align}
\end{theorem}
\begin{remark}\label{remark:xk and xi k}
We emphasize that the $x_{k}$'s are not exactly the points of the Pearcey process, but rather they are the points of the ``absolute value" of the Pearcey process. More precisely, if $\xi_{1},\xi_{2},\ldots$ are the points of the Pearcey process ordered such that $|\xi_{1}| < |\xi_{2}| < \ldots$, then $x_{k} = |\xi_{k}|$.
\end{remark}
The CLT \eqref{convergence a la gustavsson} is a rather straightforward consequence of the following CLT obtained by Dai, Xu and Zhang \cite[Corollary 2.4]{DXZ2020 thinning}:
\begin{align}\label{CLT DXZ}
\frac{N(s)-\big( \frac{3\sqrt{3}}{4\pi}s^{\frac{4}{3}}-\frac{\sqrt{3}\rho}{2\pi}s^{\frac{2}{3}} \big)}{\frac{2\sqrt{\log s}}{\sqrt{3}\pi}} \overset{d}{\longrightarrow} \mathcal{N}(0,1) \qquad \mbox{as} \qquad s \to + \infty.
\end{align}
The CLTs \eqref{convergence a la gustavsson} and \eqref{CLT DXZ} give information about the fluctuations of a single point around its classical location. Similar CLTs exist for various other point processes, see e.g. \cite{Johansson98, Gustavsson}. On the other hand, \eqref{upper bound rigidity 1}--\eqref{upper bound rigidity 2} are upper bounds for the maximum (properly rescaled) fluctuations of the points; they give information about the global rigidity of the Pearcey process. Over the past few years, we have witnessed significant progress in understanding the global rigidity of various point processes, see e.g. \cite{ErdosYauYin, ArguinBeliusBourgade, ChhaibiMadauleNajnudel, HolcombPaquette, PaquetteZeitouni, LambertCircular, CFLW, ChCl4, CGMY}, and this note can be viewed as a modest contribution to this ongoing effort. The global rigidity upper bounds \eqref{upper bound rigidity 1}--\eqref{upper bound rigidity 2} follow (almost) directly from two recent results: a global rigidity theorem from \cite{ChCl4} and the first exponential moment asymptotics for the Pearcey process from \cite{DXZ2020 thinning}. This fact has remained unnoticed until now probably because the general result of \cite{ChCl4} applies to point processes with almost surely a smallest particle, and that the Pearcey process does not meet this criteria. Our simple idea in this note is to apply \cite{ChCl4} to the ``absolute value" of the Pearcey process. 

\medskip On the supposition that upper bounds obtained via \cite{ChCl4} are sharp, see in particular \cite[Remark 1.3]{ChCl4} and Figures \ref{fig:rigidity of the counting function} and \ref{fig:rigidity of the points} below, we also conjecture the following global rigidity lower bounds.
\begin{conjecture}\label{conj:rigidity Pearcey}
For any $\epsilon > 0$, we have
\begin{align}
& \lim_{s \to \infty}\mathbb P\left(\sup_{x> s}\left|\frac{N(x)-\big( \frac{3\sqrt{3}}{4\pi}x^{\frac{4}{3}}-\frac{\sqrt{3}\rho}{2\pi}x^{\frac{2}{3}} \big)}{\log x}\right| \geq  \frac{4\sqrt{2}}{3\pi} - \epsilon \right) = 1, \label{lower bound rigidity 1} \\
& \lim_{k_0\to\infty}\mathbb P\left( \sup_{k \geq k_{0}} \frac{|\frac{3\sqrt{3}}{4\pi}x_k^{\frac{4}{3}}-\frac{\sqrt{3}\rho}{2\pi}x_k^{\frac{2}{3}}-k|}{\log k}\geq \frac{\sqrt{2}}{\pi} - \epsilon\right)=1. \label{lower bound rigidity 2}
\end{align}
\end{conjecture}
\begin{remark}
For several point processes in random matrix theory, the so-called second exponential moment asymptotics have turned out to be important in obtaining optimal rigidity lower bounds, see e.g. \cite{ArguinBeliusBourgade, ChhaibiMadauleNajnudel, PaquetteZeitouni, CFLW}. Therefore, the second exponential moment asymptotics of the Pearcey process are expected to be relevant in proving (or disproving) Conjecture \ref{conj:rigidity Pearcey}. These asymptotics are currently not available in the literature.

 %that the large $s$ asymptotics of $\mathbb E\big[e^{\gamma_{1} N(u_{1}s)+\gamma_{2} N(u_{2}s)}\big]$ with $u_{1},u_{2}>0$, $\gamma_{1},\gamma_{2} \in \mathbb{R}$ are relevant to prove (or disprove) 

%Second exponential moment asymptotics of the Pearcey process correspond to 

%As it turns out, the second exponential moment asymptotics have been an important tool  Therefore, we expect that here too, 

%By analogy with ,  However, these asymptotics are currently not available in the literature. %The second exponential moment asymptotics for the Pearcey process, which are the large $s$ asymptotics of $\mathbb E\big[e^{\gamma_{1} N(u_{1}s)+\gamma_{2} N(u_{2}s)}\big]$ with $u_{1},u_{2}>0$, $\gamma_{1},\gamma_{2} \in \mathbb{R}$, are currently not available in the literature.
\end{remark}

\subsection*{Proofs of the main results}

\begin{proof}[Proof of \eqref{convergence a la gustavsson}]
The proof mainly follows Gustavsson \cite[Theorem 1.2]{Gustavsson}. For $s > \frac{|\rho|^{\frac{3}{2}}}{3\sqrt{3}}$, define
\begin{align}\label{mu sigma2 Pearcey}
\mu(s)=\frac{3\sqrt{3}}{4\pi}s^{\frac{4}{3}}-\frac{\sqrt{3}\rho}{2\pi}s^{\frac{2}{3}},\qquad \sigma^2(s)=\frac{4}{3\pi^2}\log s.
\end{align}
It is easy to verify that the function $\mu$ is strictly increasing on its domain of definition, and hence is invertible. Since 
\begin{align*}
\sqrt{\sigma^{2} \circ \mu^{-1}(k)} = \frac{\sqrt{\log k}}{\pi}(1+o(1)), \qquad \mbox{as } k \to + \infty,
\end{align*}
\eqref{convergence a la gustavsson} follows if we prove that 
\begin{align}\label{lol1}
\frac{\mu(x_{k})-k}{\sqrt{\sigma^{2} \circ \mu^{-1}(k)}} \quad  \overset{d}{\longrightarrow} \quad  \mathcal{N}(0,1), \qquad \mbox{as } k \to + \infty.
\end{align}
Let $y \in \mathbb{R}$ be an arbitrary constant. For all sufficiently large $k$, we define $s_{k} = \mu^{-1}\Big(k + y\sqrt{\sigma^{2}\circ \mu^{-1}(k)}\Big)$. We have
\begin{align*}
& \mathbb{P}\Big[ \frac{\mu(x_{k})-k}{\sqrt{\sigma^{2} \circ \mu^{-1}(k)}} \leq y \Big] = \mathbb{P}\Big[N(s_{k}) \geq k \Big] = \mathbb{P}\Big[\frac{N(s_{k})-\mu(s_{k})}{\sqrt{\sigma^{2}(s_{k})}} \geq \frac{k-\mu(s_{k})}{\sqrt{\sigma^{2}(s_{k})}} \Big] \\
& = \mathbb{P}\Big[\frac{\mu(s_{k})-N(s_{k})}{\sqrt{\sigma^{2}(s_{k})}} \leq y\frac{\sqrt{\sigma^{2}\circ \mu^{-1}(k)}}{\sqrt{\sigma^{2}(s_{k})}} \Big] = \mathbb{P}\Big[\frac{\mu(s_{k})-N(s_{k})}{\sqrt{\sigma^{2}(s_{k})}} \leq y(1+o(1)) \Big], \qquad \mbox{as } k \to + \infty.
\end{align*}
The CLT \eqref{lol1}, and hence \eqref{convergence a la gustavsson}, now follows directly from \eqref{CLT DXZ}.
\end{proof}

As already mentioned, the upper bounds \eqref{upper bound rigidity 1} and \eqref{upper bound rigidity 2} are rather direct consequences of two recent results from \cite{DXZ2020 thinning} and \cite{ChCl4}. Let us briefly recall these results.

\begin{theorem}(First exponential moment asymptotics from \cite[Theorem 2.3]{DXZ2020 thinning}). \label{thm:exponential moments}
We have
\begin{align*}
\mathbb E\big[e^{\gamma N(s)}\big]=C(\gamma)e^{\gamma\mu(s)+\frac{\gamma^2}{2}\sigma^2(s)}(1+\bigO(s^{-2/3})), \qquad \mbox{as } s \to + \infty
\end{align*}
uniformly for $\gamma$ in compact subsets of $\mathbb{R}$, where $\mu, \sigma^{2}$ are given by \eqref{mu sigma2 Pearcey} and $C(\gamma)$ is independent of $s$ and continuous as a function of $\gamma$.
\end{theorem}
\begin{remark}
$C(\gamma)$ has in fact also been obtained in \cite{DXZ2020 thinning}, but this is not needed for us.
\end{remark}
\begin{theorem}(\cite[Lemma 2.1 and Theorem 1.2]{ChCl4}) \label{thm: rigidity with Claeys}
Let $X$ be a locally finite random point configuration on $\mathbb{R}$ distributed according to a given point process. Assume that $X$ has almost surely a smallest particle, let $x_{k}$ denote the $k$-th smallest point of $X$, and let $\widetilde{N}(x)$ be the random variable that counts the number of points in $X$ that are $\leq x$. Assume that there exist constants $\mathrm{C}, a >0$, $s_0\in\mathbb R$, $M > \sqrt{2/a}$ and functions $\tilde{\mu},\tilde{\sigma}:[s_0,+\infty)\to [0,+\infty)$ such that the following holds:
\begin{enumerate}
\item[(1)] We have
\begin{equation}\label{expmomentbound}
\mathbb{E} \big[e^{\gamma \widetilde{N}(s)}\big]\leq \mathrm{C} \, e^{\gamma \tilde{\mu}(s)+\frac{\gamma^{2}}{2}\tilde{\sigma}^2(s)},
\end{equation}
for all $\gamma\in[-M,M]$ and for all $s>s_0$.
\item[(2)] The functions $\tilde{\mu}$ and $\tilde{\sigma}$ are strictly increasing and differentiable, and they satisfy 
\begin{align*}
\lim_{s\to + \infty} \tilde{\mu}(s) = + \infty, \qquad  \lim_{s\to + \infty} \tilde{\sigma}(s) = + \infty, \qquad \lim_{s\to+\infty}\frac{s\tilde{\mu}'(s)}{\tilde{\sigma}^2(s)}=+\infty.
\end{align*}
Moreover, $s\mapsto s\tilde{\mu}'(s)$ is weakly increasing, $\tilde{\sigma}^2\circ\tilde{\mu}^{-1}:[\tilde{\mu}(s_0),+\infty)\to [0,+\infty)$ is strictly concave and $(\tilde{\sigma}^2\circ\tilde{\mu}^{-1})(s)\sim a\log s \mbox{ as } s\to +\infty$.
\end{enumerate}
Then, for any $\epsilon>0$, it holds that
\begin{align}
& \lim_{s \to + \infty} \mathbb{P}\left(\sup_{x>s}\left|\frac{\widetilde{N}(x)-\tilde{\mu}(x)}{\tilde{\sigma}^2(x)}\right| \leq   \sqrt{\frac{2}{a}}+\epsilon\right) = 1, \label{ChCl rig 1} \\
& \lim_{k_0\to\infty}\mathbb P\left( \sup_{k \geq k_{0}}\frac{|\tilde{\mu}(x_k)-k|}{\tilde{\sigma}^2(\tilde{\mu}^{-1}(k))}\leq \sqrt{\frac{2}{a}}+\epsilon \right)=1. \label{ChCl rig 2}
\end{align}
\end{theorem}
\begin{proof}[Proof of \eqref{upper bound rigidity 1} and \eqref{upper bound rigidity 2}]
Let $Y=\{\xi_{k}\}_{k \geq 1}$ be a random point configuration distributed according to the Pearcey process. Since $Y$ does not have almost surely a smallest particle, we cannot apply Theorem \ref{thm: rigidity with Claeys} to the Pearcey process. Instead we define
\begin{align*}
X = \{|\xi| : \; \xi  \in Y \}.
\end{align*}
Since $Y$ is locally finite, $X$ has a smallest particle. Note that $X$ is distributed according to a point process which is not determinantal, but that does not matter for Theorem \ref{thm: rigidity with Claeys}. Let $\widetilde{N}(x)$ denote the number of points in $X$ that are in the interval $[0,x]$. Using Theorem \ref{thm:exponential moments}, it is directly seen that $X$ verifies the assumptions of Theorem \ref{thm: rigidity with Claeys} with 
\begin{align*}
\tilde{\mu} = \mu, \quad \tilde{\sigma} = \sigma, \quad M=10, \quad \mathrm{C} = 2 \max_{\gamma \in [-M,M]}C(\gamma), \quad a=\frac{1}{\pi^2}, \quad s_{0} \mbox{ sufficiently large}.
\end{align*}
By definition of $\widetilde{N}$, we have
\begin{align*}
\widetilde{N}(x) = \# (X \cap [0,x]) = \# (Y \cap [-x,x]) = N(x).
\end{align*}
Now \eqref{upper bound rigidity 1}--\eqref{upper bound rigidity 2} follow straightforwardly from \eqref{ChCl rig 1}--\eqref{ChCl rig 2} after substituting $\tilde{\mu}, \tilde{\sigma}, a, \widetilde{N}(x)$ by $\mu,\sigma,\frac{1}{\pi^{2}}$, $N(x)$, respectively. 
\end{proof}

\subsection*{Numerical support for Theorem \ref{thm:rigidity Pearcey} and Conjecture \ref{conj:rigidity Pearcey}}
We provide here some numerical data to support the validity of Theorem \ref{thm:rigidity Pearcey} and Conjecture \ref{conj:rigidity Pearcey}. For this, we use the result \cite[Theorem 1.1]{BleherKuijlaarsIII} of Bleher and Kuijlaars which states that the (properly rescaled) eigenvalues around $0$ of a large Gaussian random matrix with an external source approximate the random points of the Pearcey process. More precisely, consider the random matrix ensemble
\begin{align*}
\frac{1}{Z_{n}}e^{-n\mathrm{Tr}(\frac{1}{2}M^{2}-AM)}dM, \qquad n \in \mathbb{N}_{>0}, \; n \mbox{ even},
\end{align*}
defined on the set $\{M\}$ of all $n \times n$ Hermitian matrices, where $A$ is a diagonal matrix with two eigenvalues $\pm (1+\frac{\rho}{2\sqrt{n}})$ of equal multiplicities. The eigenvalues $\lambda_{1}^{(n)},\ldots,\lambda_{n}^{(n)}$ of $M$ are distributed according to a determinantal point process on $\mathbb{R}$ whose kernel $K_{n,\rho}$ satisfies (see \cite[Theorem 1.1]{BleherKuijlaarsIII}):
\begin{align}\label{convergence of kernel}
\lim_{n \to + \infty}\frac{1}{n^{3/4}} K_{n,\rho}\Big(\frac{x}{n^{3/4}},\frac{y}{n^{3/4}}\Big) = K_{\rho}^{\mathrm{Pe}}(x,y), \qquad x,y \in \mathbb{R}.
\end{align}
Let $N_{n}(x) := \#\{\lambda_{j}^{(n)} : \lambda_{j}^{(n)} \in (-\frac{x}{n^{3/4}},\frac{x}{n^{3/4}})\}$. One expects from \eqref{convergence of kernel} that $N_{n}(x)$ converges in distribution to $N(x)$ as $n \to + \infty$. Let us choose the numbering of the $\lambda_{j}^{(n)}$'s such that $|\lambda_{1}^{(n)}| \leq \ldots \leq |\lambda_{n}^{(n)}|$. From \eqref{convergence of kernel}, one also expects that for any fixed $k \in \mathbb{N}_{>0}$, the random variables $n^{3/4}\lambda_{1}^{(n)},\ldots,n^{3/4}\lambda_{k}^{(n)}$ converge in distribution to $x_{1},\ldots,x_{k}$ as $n \to + \infty$.

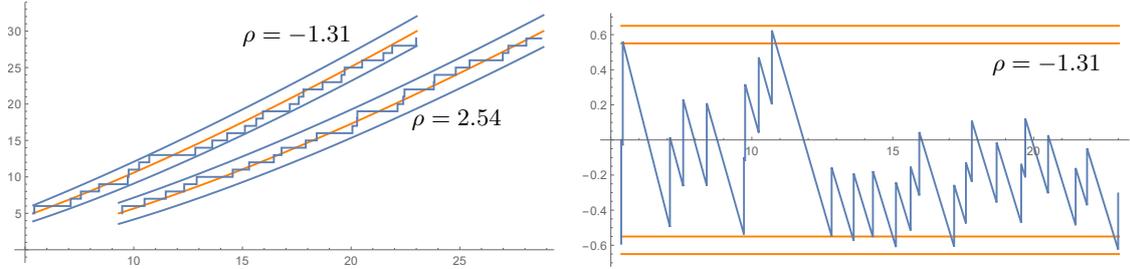
\begin{figure}[h]
\begin{center}
\begin{tikzpicture}
\node at (0,0) {\includegraphics[scale=0.3]{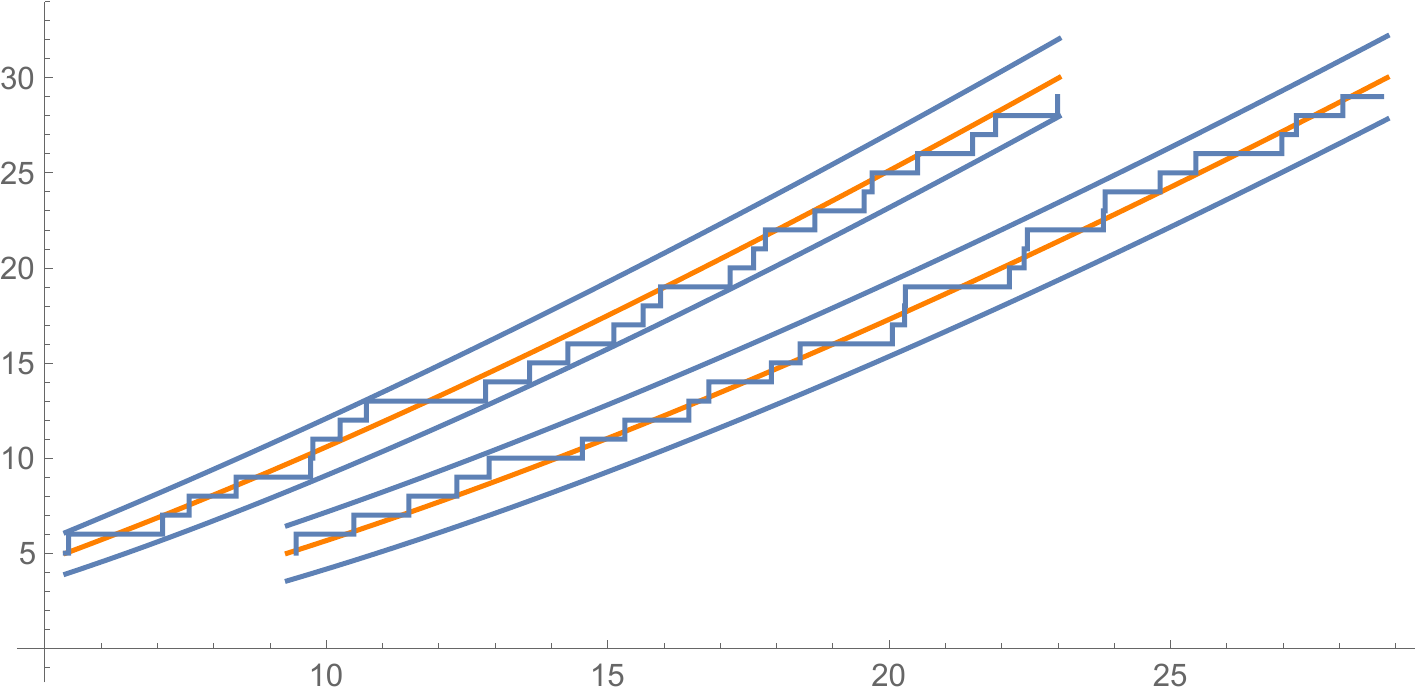}};
%\node at (0,-2.5) {Wright's generalized Bessel};
\node at (0.2,1.3) {\small $\rho=-1.31$};
\node at (2.3,0.2) {\small $\rho=2.54$};
\end{tikzpicture}
\begin{tikzpicture}
\node at (0,0) {\includegraphics[scale=0.3]{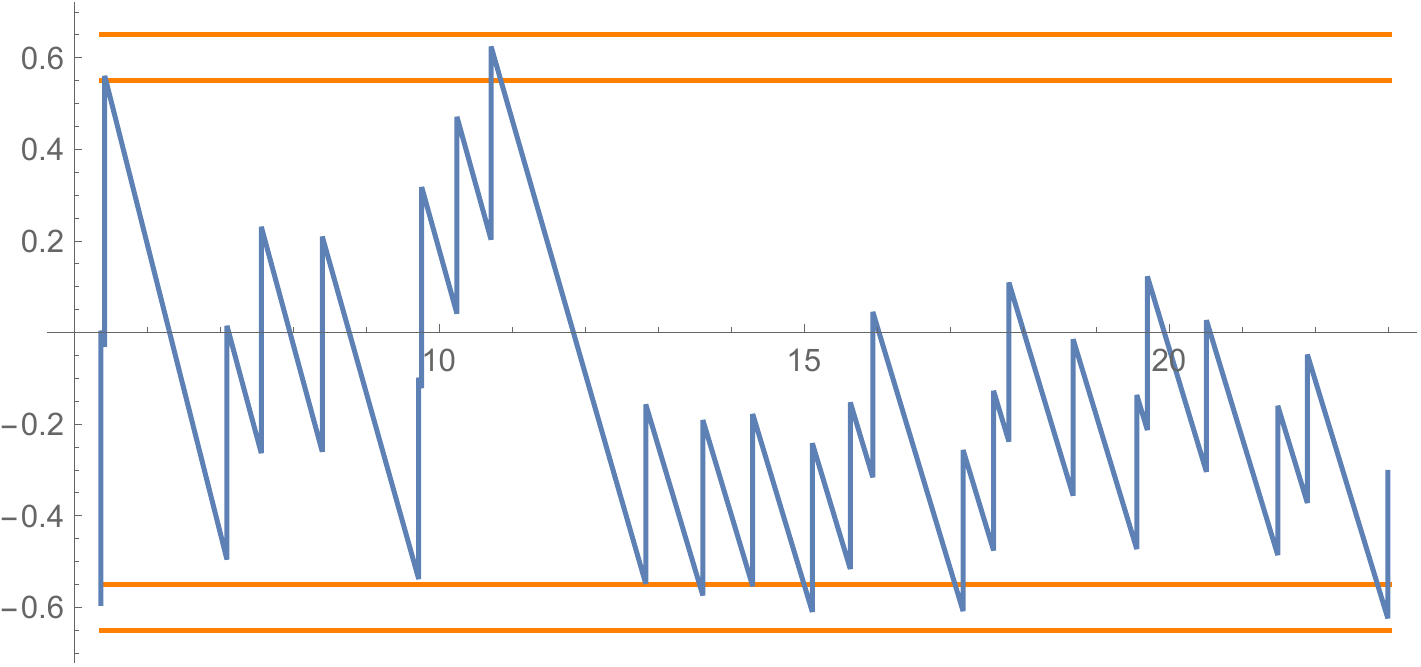}};
\node at (2.5,1) {\small $\rho=-1.31$};
%\node at (2.6,-0.4) {$r=1$};
%\node at (0,-2.5) {Meijer-$G$};
\end{tikzpicture}
\end{center}
\vspace{-0.5cm}\caption{\label{fig:rigidity of the counting function}Numerical support for \eqref{upper bound rigidity 1} and \eqref{lower bound rigidity 1}.} 
\end{figure}

\medskip The probabilistic bound \eqref{upper bound rigidity 1} implies that for any $\epsilon >0$, the probability that
\begin{align}\label{rewriting of thm 1}
\mu(x)- \Big( \frac{4\sqrt{2}}{3\pi}+\epsilon \Big) \log x \leq N(x) \leq \mu(x) + \Big( \frac{4\sqrt{2}}{3\pi}+\epsilon \Big) \log x \qquad \mbox{for all }x>s
\end{align}
tends to $1$ as $s \to + \infty$. 
In Figure \ref{fig:rigidity of the counting function} (left), the blue step-like curves represent some graphs of the random function $x \mapsto N_{n}(x)$ for $n=400$ and two values of $\rho$, the blue smooth curves are the upper and lower bounds in \eqref{rewriting of thm 1} with $\epsilon=0.05$, and the orange curves represent $x\mapsto \mu(x)$. In Figure \ref{fig:rigidity of the counting function} (right), the orange lines indicate the heights $\pm \frac{4\sqrt{2}}{3\pi} \pm \epsilon$ with $\epsilon=0.05$, and the blue curve is a graph of the random function
\begin{align*}
x \mapsto \frac{N_{n}(x)-\mu(x)}{\log x},
\end{align*}
for $n=400$ and $\rho=-1.31$. We see that some of the local minimums and maximums of this function are in the bands between the orange lines, which supports the validity of both \eqref{upper bound rigidity 1} and \eqref{lower bound rigidity 1}.
\begin{figure}[h]
\begin{center}
\begin{tikzpicture}
\node at (0,0) {\includegraphics[scale=0.3]{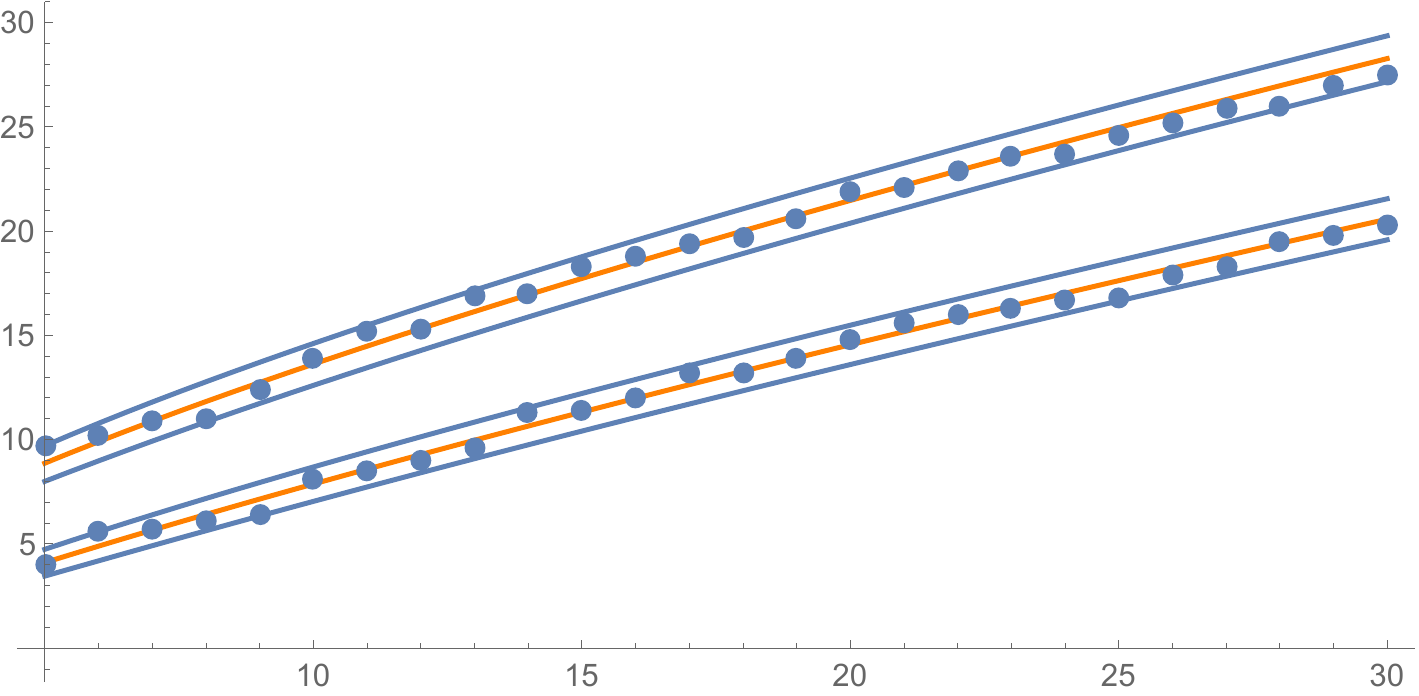}};
%\node at (0,-2.5) {Wright's generalized Bessel};
\node at (1.3,1.5) {\small $\rho=2.19$};
\node at (2,-0.1) {\small $\rho=-3.24$};
\end{tikzpicture}
\begin{tikzpicture}
\node at (0,0) {\includegraphics[scale=0.3]{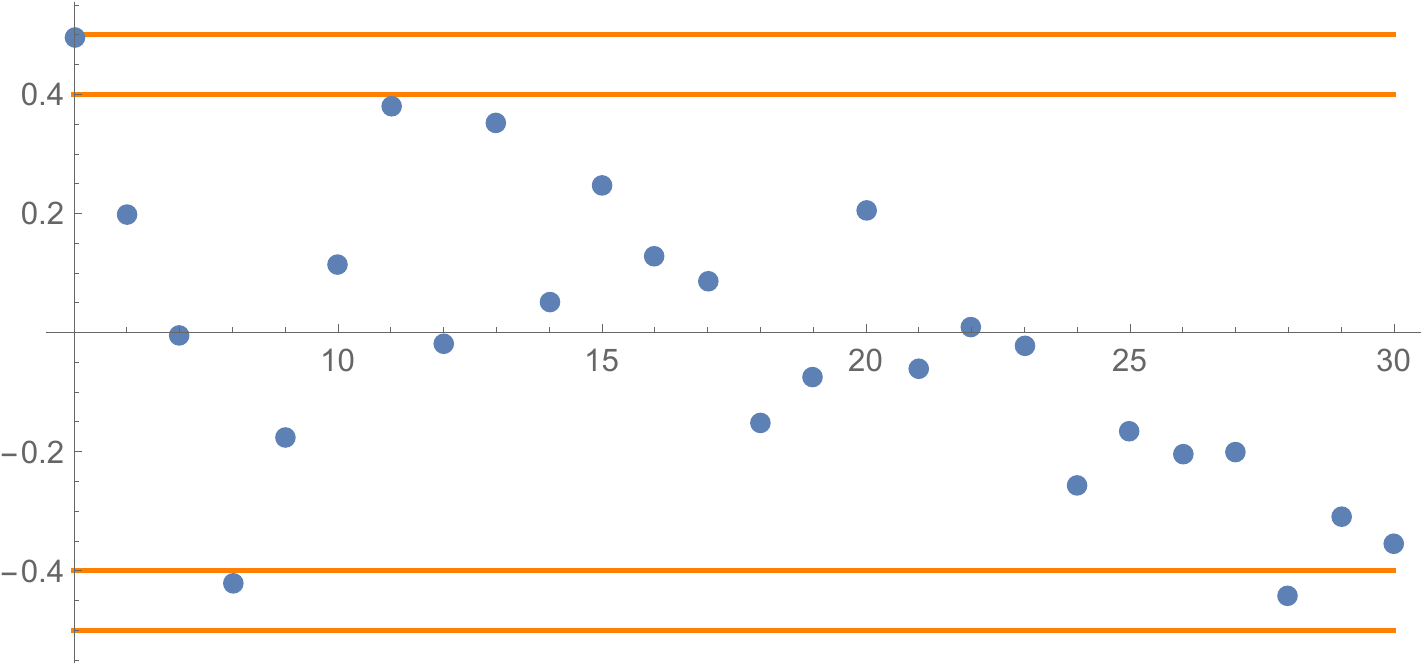}};
\node at (2.7,0.9) {\small $\rho=2.19$};
%\node at (2.6,-0.4) {$r=1$};
%\node at (0,-2.5) {Meijer-$G$};
\end{tikzpicture}
\end{center}
\vspace{-0.5cm}\caption{\label{fig:rigidity of the points}
 Numerical support for \eqref{upper bound rigidity 2} and \eqref{lower bound rigidity 2}.} 
\end{figure}
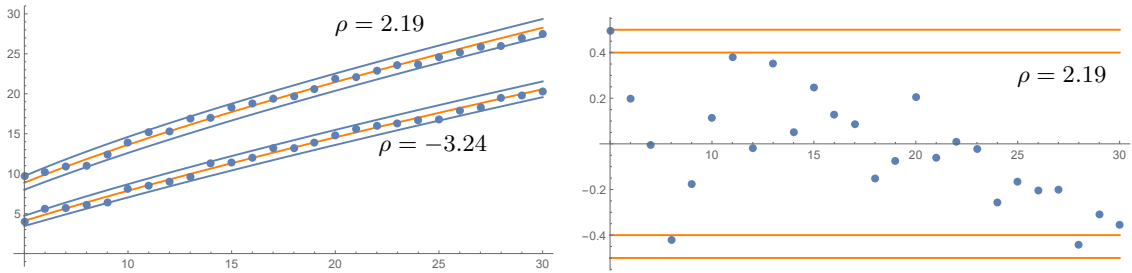

\medskip

The result \eqref{upper bound rigidity 2} implies that for any $\epsilon >0$, the probability that
\begin{align}\label{rewriting of thm 2}
\mu^{-1}\bigg( k - \Big( \frac{\sqrt{2}}{\pi}+\epsilon \Big) \log k \bigg) \leq x_{k} \leq \mu^{-1}\bigg( k + \Big( \frac{\sqrt{2}}{\pi}+\epsilon \Big) \log k \bigg) \qquad \mbox{for all } k \geq k_{0}
\end{align}
tends to $1$ as $k_{0} \to + \infty$. In Figure \ref{fig:rigidity of the points} (left), the blue dots represent the random points $(k,n^{3/4}\lambda_{k}^{(n)})$ for $n=400$ and two values of $\rho$, the blue curves are the upper and lower bounds in \eqref{rewriting of thm 2} with $\epsilon=0.05$, and the orange curves represent
\begin{align*}
k \mapsto \mu^{-1}(k) = \bigg( \frac{1}{3}\Big( \rho + \sqrt{4\sqrt{3} \pi k + \rho^{2}} \Big) \bigg)^{3/2}.
\end{align*}
In Figure \ref{fig:rigidity of the points} (right), the orange lines indicate the heights $\pm \frac{\sqrt{2}}{\pi} \pm \epsilon$ with $\epsilon=0.05$, and the blue dots have coordinates
\begin{align*}
\bigg(k,\frac{\mu(n^{3/4}\lambda_{k}^{(n)})-k}{\log k} \bigg),
\end{align*}
for $n=400$ and $\rho=2.19$. We see that some of these points are in the bands between the orange lines, which supports the validity of both \eqref{upper bound rigidity 2} and \eqref{lower bound rigidity 2}.

\paragraph{Acknowledgements.}
The author was supported by the European Research Council, Grant Agreement No. 682537.

\end{document}